\documentclass[12pt]{amsart}
\usepackage{amssymb,amsfonts,amsmath,amsthm,amscd,amssymb,latexsym}

\newtheorem{theorem}{Theorem}
\newtheorem{corollary}{Corollary}

\newtheorem{lemma}{Lemma}

\title{An extremal property of trigonometric polynomials}
\author{ D.V. Dmitrishin and A.D. Khamitova}
\address{ Odessa National University and Odessa National Polytechnic University} 
\email{ddmitrishin@yahoo.com}  


\begin{document}

\begin{abstract}
In this article the solution of the special problem of the conditional extremum for the conjugate trigonometric polynomials is given. 
$$\mathop{\sup }\limits_{\begin{array}{c} {(a_{1} ,\ldots,a_{n} )} \\ {\sum a_{j} =1 } \end{array}} \mathop{\min }\limits_{t} \left\{\, \sum _{j=1}^{n}a_{j} \cos jt:\sum _{j=1}^{n}a_{j} \sin jt=0  \, \right\}=-tg^{2} \frac{\pi }{2(n+1)} .$$
A possibility to apply  this result to the problems of optimal stabilization of quasidynamic chaos in discrete systems is mentioned.
\end{abstract}
\maketitle

\bigskip
As we know, the problem of optimal impact on chaotic regime is one of the most fundamental in nonlinear dynamics \cite{CY}. For the multiparameter families of discrete systems this problem can be reduced to the choice of a direction that provides maximum stability in one-parameter space.  When we change this parameter we can explore the sequence of bifurcations that leads to occurrence of a chaotic attractor.

\bigskip
First, bifurcation values of the parameter corresponds to the loss of stable equilibrium position in the system. These values are related to the area of Schur stability of family of polynomials,
\begin{equation} \label{GrindEQ__1_} 
\left\{f(\lambda )=\lambda ^{n} +k(a_{1} \lambda ^{n-1} +\ldots+a_{n} ),\qquad \sum _{j=1}^{n}a_{j} =1 \right\} 
\end{equation} 
where $k$ is a parameter. All polynomials of the family \eqref{GrindEQ__1_} are stable for $k=0$. Moreover, there exist two positive constants $k_{1} ,k_{2} $ which depend on the coefficients $a_{1} ,\ldots,a_{n} $, such that the family remains stable for $k\in (-k_{1} ,k_{2} )$ and the stability disturbs if $k=k_{2} +\varepsilon $ or $k=-k_{1} -\varepsilon $, $\forall\varepsilon>0$. We need to maximize the length of robust stability segment i.e. the function
\begin{equation} \label{GrindEQ__2_}
\Phi (a_{1} ,\ldots,a_{n} )=k_{1} (a_{1} ,\ldots,a_{n} )+k_{2} (a_{1} ,\ldots,a_{n} ).                                      
\end{equation} 

Function $\Phi (a_{1} ,\ldots,a_{n} )$ has simple geometrical meaning as
$$\frac{f(e^{it} )}{ke^{int} } = \frac{1}{k} +\sum _{j=1}^{n}a_{j} \cos jt - i\sum _{j=1}^{n}a_{j} \sin jt.$$ 
The points of intersection of curve $$\left\{ x=\sum _{j=1}^{n}a_{j} \cos jt, y=-\sum _{j=1}^{n}a_{j} \sin jt  \right\} $$ on the $OXY$ plane with $OX$ axis correspond to those values of parameter $k$, for which the polynomial of the family \eqref{GrindEQ__1_} has zeros on the unit circle. The length of the longest segment which is defined by these points of intersection is $\frac{1}{k_{1} } +\frac{1}{k_{2} } .$

\bigskip
Since $\sum _{j=1}^{n}a_{j} =1 $, $f\eqref{GrindEQ__1_}=1+k$. Hence, $\mathop{\max }\limits_{a_{1} ,\ldots,a_{n} } \left\{\, k_{1} (a_{1} ,\ldots,a_{n} )\, \right\}\le 1$. 

Let
\begin{equation}\label{aboutI}
I=\mathop{\sup }\limits_{a_{1} ,\ldots,a_{n} } \mathop{\min }\limits_{t} \left\{\, C(t):S(t)=0,\sum _{j=1}^{n}a_{j} =1\,  \right\}\end{equation}
We have that
$$\mathop{\max }\limits_{a_{1} ,\ldots,a_{n} } \Phi (a_{1} ,\ldots,a_{n} )\le 1-\frac{1}{I}. $$
Note, that the value of $I$ is negative as it follows from the Lemma 1. 

\bigskip
In order to find the value of \textit{I}, we need the results that are included in the following six lemmas.

Let
\begin{equation} \label{GrindEQ__3_} 
\rho =\mathop{\min }\limits_{t} \left\{\, C(t):S(t)=0,\qquad\sum _{j=1}^{n}a_{j} =1\,  \right\} .                                               
\end{equation} 
Since the polynomials $C(t)$ and $S(t)$ are trigonometric, it is enough to consider the minimum in \eqref{GrindEQ__3_} on the segment $ \left[ 0, \pi \right]$.

\bigskip
\begin{lemma}\label{lem1} The value of $\rho $ is negative.\end{lemma}

\begin{proof} 
Let function $F(z)=\sum _{j=1}^{n}a_{j} z^{j}$, where $z$ is a complex variable. It is clear that $z_{0} =0$ is a zero of $F(z)$.

Hence, when the point $z$ goes along the unit circle once, the argument of the function increases at least by $2\pi $. This means that the plot of $x+iy=F(e^{it} )$ on the plane $OXY$ with $t\in [\, 0,\, 2\pi \, ]$ surrounds zero. Since the function $F(e^{it} )$ is continuous, there exists a $t_{0} \in [\, 0,\, 2\pi \, ]$ such that $Re\, \left\{\, \, F(e^{it_{0} } )\, \, \right\}=C(t_{0} )<0$, $Im\, \left\{\, \, F(e^{it_{0} } )\, \, \right\}=S(t_{0} )=0$. 

This completes the proof of Lemma~\ref{lem1}.
\end{proof}

\begin{lemma}\label{lem2} Let $S(t_{0} )=0$. Then the trigonometric polynomial $S(t)$ is presented uniquely by $S(t)=(\cos t-\cos t_{0} )\sum _{j=1}^{n-1}a_{j} '\sin jt $ and $(1-\cos t_{0} )\sum _{j=1}^{n-1}a_{j} '=1+\frac{a_{1} '}{2}  $\end{lemma}

\begin{proof}  Coefficients $a_{1} ,\ldots,a_{n} $ and $a_{1} ',\ldots,a'_{n-1} $ are connected by relations
\begin{equation} \label{GrindEQ__4_} 
\left\{\begin{array}{c} {a_{1} =-\cos t_{0} \cdot a_{1} '+\frac{1}{2} a_{2} ',} \\ {a_{2} =\frac{1}{2} a_{1} '-\cos t_{0} \cdot a_{2} '+\frac{1}{2} a_{3} ',} \\ {\ldots} \\ {a_{n-1} =\frac{1}{2} a_{n-2} '-\cos t_{0} \cdot a_{n-1} ',} \\ {a_{n} =\frac{1}{2} a_{n-1} '} \end{array}\right.  
\end{equation} 
The determinant of this system modulo first equation is equal to $2{}^{1-n} $. So, the coefficients $a_{1} ',\ldots,a_{n-1} '$ are uniquely presented via the coefficients $a_{2} ,\ldots,a_{n} $. Thus, from the normalization requirements we obtain:
$$1=\sum _{j=1}^{n}a_{j} =-\cos t_{0} \cdot \sum _{j=1}^{n-1}a_{j} '  +\sum _{j=2}^{n-1}a_{j} '=(1-\cos t_{0} )\cdot \sum _{j=1}^{n-1}a_{j} ' -\frac{1}{2} a_{1} ' .$$

This proves Lemma \ref{lem2}.
\end{proof}

\begin{lemma}\label{lem3}  Let $S(t_{0} )=0,\, \, t_{0} \in (\, 0,\pi \, )$. Then 
$$C(t)=-\frac{a_{1} '}{2} +(\cos t-\cos t_{0} )\sum _{j=1}^{n-1}a_{j} '\cos jt.$$
So $C(t_{0} )=-\frac{a_{1} '}{2} $, where the coefficients $a_{1} ',\, \, \ldots \, \, ,a_{n-1} '$ are determined via the coefficients $a_{2} ,\ldots,a_{n} $ in the system of equations \eqref{GrindEQ__4_}.
\end{lemma}

\begin{proof} It follows from the system \eqref{GrindEQ__4_} that

$C_{n} (t)=\sum _{j=1}^{n}a_{j} \cos jt$ =
$ (-\cos t_{0}a_{1} '+\frac{1}{2} a_{2} ' )\cos t$ + 
$\ldots$+$(\frac{1}{2} a_{n-3} '-\cos t_{0}  a_{n-2} ')\cos (n-2)t $+ $(\frac{1}{2} a_{n-2} '-\cos t_{0}  a_{n-1} ') \cos (n-1)t$+$\frac{1}{2} a_{n-1} '\cos nt$ = $C_{n-1} (t)+\frac{1}{2} a_{n-1} '(\cos (n-2)t-2\cos t_{0} \cos (n-1)t$+$\cos nt$ = $C_{n-1} (t)+a_{n-1} ' \cos (n-1)t (\cos t-\cos t_{0} )$ = $C_{2} (t)+(\cos t-\cos t_{0} ) \sum _{j=2}^{n-1}a_{j} '\cos jt\,$ =$-\cos t_{0}  a_{1} ' \cos t$+$\frac{a_{1} '}{2} \cos 2t$+$(\cos t-\cos t_{0} ) \sum _{j=2}^{n-1}a_{j} '\cos jt$ = $-\frac{a_{1} '}{2} $+$(\cos t-\cos t_{0} ) \sum _{j=1}^{n-1}a_{j} '\cos jt. $

Hence $C_{n} (t_{0} )=-\frac{a_{1} '}{2} $.\end{proof}

\begin{lemma}\label{lem4} Let $S(t_{0} )=0,t_{0} \in (\, 0,\, \pi \, )$. Then
\[C(\pi )=-(1+\cos t_{0} )(-a_{1} '+a_{2} '-\ldots)-\frac{a_{1} '}{2} .\] \end{lemma}

Lemma \ref{lem4} is a consequence of Lemma \ref{lem3}. 

\begin{lemma}\label{lem5} Let $S(t_{0} )=S(t_{1} )=0,\; t_{0} \in (\, 0,\, \pi \, ),\; t_{1} \in (\, 0,\, \pi \, ),\; t_{0} \ne t_{1} $. Then $S(t)$ is presented uniquely in the next form:
$$S(t)=(\cos t-\cos t_{0} )(\cos t-\cos t_{1} )\sum _{j=1}^{n-2}a_{j} '' \sin jt,$$ and
$$(1-\cos t_{0} )(1-\cos t_{1} )\sum _{j=1}^{n-2}a_{j} '' -(1-\cos t_{0} -\cos t_{1} )\frac{a_{1} ''}{2} -\frac{a_{2} ''}{4} =1.$$\end{lemma}

 For the proof we need to apply the Lemma 2 twice.
 
\begin{lemma}\label{lem6} Let $S(t_{0} )=S(t_{1} )=0,\; t_{0} \in (\, 0,\, \pi \, ),\; t_{1} \in (\, 0,\, \pi \, ),\; t_{0} \ne t_{1} $. Then
$$C(t_{0} )=-\frac{a_{2} ''}{4} +\frac{a_{1} ''}{2} \cos t_{1} ,$$
$$C(t_{1} )=-\frac{a_{2} ''}{4} +\frac{a_{1} ''}{2} \cos t_{0} .$$
\end{lemma}

\begin{proof} By Lemma \ref{lem2}
$$S(t)=(\cos t-\cos t_{0} )\cdot \sum _{j=1}^{n-1}a_{j} '\sin jt,$$ 
where $a_{1} '=\frac{a_{2} ''}{2} -a_{1} ''\cos t_{1} $ by Lemma \ref{lem5}. Then by Lemma \ref{lem3}
$$C(t_{0} )=-\frac{a_{1} '}{2} =-\frac{a_{2} ''}{4} +\frac{a_{1} ''}{2} \cos t_{1} .$$

We get the value of $C(t_{1} )$ in a similarly way. \end{proof}

\begin{corollary} If $S(t_{0} )=S(t_{1} )=0$, $C(t_{0} )=C(t_{1} )$, $t_{0} \ne t_{1} $, then
$$C(t_{0} )=C(t_{1} )=-\frac{a_{2} ''}{4} .$$
\end{corollary}

Now we come to the main result of the paper:
\begin{theorem} 
Let $C(t)$, $S(t)$ be the pair of conjugate trigonometric polynomials defined by
$$C(t)=\sum _{j=1}^{n}a_{j} \cos jt, \ \ S(t) = \sum _{j=1}^{n}a_{j} \sin jt,$$
satisfing the normalization condition $\sum _{j=1}^{n}a_{j} =1 .$ Denote by $I$ the conditional extremum $\mathop{\sup }\limits_{a_{1},      \ldots ,a_{n} } \mathop{\min }\limits_{t} \left\{\, \, C(t):S(t)=0\, \right\}$, we have that 
\begin{equation} I=-tg^{2} \frac{\pi }{2(n+1)}.\end{equation}
\end{theorem} 

\begin{proof} The function $S(t)$ vanishes at $t=\pi$ for all coefficients $a_{1} ,\ldots,a_{n} $. In the following, 
we will find the value $\sup\rho (a_{1} ,..,a_{n} )\,$ on the set 
$$A_{R} =\left\{\left(a_{1} ,\ldots,a_{n} \right) : \sum _{j=1}^{n}a_{j} =1,\sum _{j=1}^{n}\left|a_{j} \right|\le R  \right\}.$$ The function $\rho (a_{1} ,..,a_{n} )$ is continuous on the set 
$A_{R} $, except those points $\left(\, a_{1} ,\, \ldots\, ,a_{n} \right)$ at which the minimum $C(t)$ is achieved at the zero of $S(t)$, where the function $S(t)$  does not change the sign.

Note that the lower limit of the function $\rho (a_{1} ,..,a_{n} )$ at the points of discontinuity is the value of the function, this means that on the set $A_{R} $ the function $\rho (a_{1} ,..,a_{n})$ is lower semicontinuous.

Together with the function $\rho (a_{1} ,..,a_{n} )$ we consider the function
$$\rho _{1} (a_{1} ,..,a_{n} )=\mathop{\min }\limits_{t\in \left[\, 0,\, \pi \, \right]} \left\{\, \, C(t) : t=T\bigcup \{\pi\} \, \right\},$$
where \textit{T} is set of points of interval $(\, 0,\, \pi \, )$ at which the function $S(t)$ changes the sign. The set $T\bigcup \{\pi\} $ is a subset of the set of all zeros of $S(t)$. Therefore, $\overline{\rho }\le \overline{\rho _{1} }$, where $\overline{\rho }$ and $\overline{\rho _{1} }$ are given, respectively, by $\overline{\rho }\le \mathop{\sup }\limits_{(a_{1} ,\ldots,a_{n})\in A_{R} } \left\{\, \, \rho (a_{1} ,\ldots,a_{n} )\, \, \right\},$\quad $\overline{\rho _{1} }\le \mathop{\sup }\limits_{(a_{1} ,\ldots,a_{n} )\in A_{R} } \left\{\, \, \rho _{1} (a_{1} ,\ldots,a_{n} )\, \, \right\}$.

\bigskip
The function $\rho _{1} (a_{1} ,\ldots,a_{n} )$ is upper semicontinuous, so on the set $A_{R} $ it reaches its maximum value, i.e. $\overline{\rho _{1} }=\mathop{\max }\limits_{(a_{1} ,\ldots,a_{n})\in A_{R} } \left\{\, \, \rho _{1} (a_{1} ,...,a_{n} )\, \, \right\}$.

A pair of trigonometric polynomials $(C^{0} (t),S^{0} (t))$ on which the maximum is reached is called an optimal pair.

Define for the optimal polynomial $S^{0} (t)$ the set $T=\left\{\, \, t_{0} ,t_{1} ,\, ...\, ,t_{q} \, \right\}$, where $0\le q\le n-2$. Besides, let
$$\min \left\{\, \, C^{0} (t_{0} ),\ldots,C^{0} (t_{q} )\, \right\}=C^{0} (t_{0} ),$$
and $C^{0} (t_{0} )<C^{0} (t_{j} )$, $j=1,\ldots, q,$ $C^{0} (t_{0} )<C^{0} (\pi )$.

Then by Lemma \ref{lem2}
\[S^{0} (t)=(\cos t-\cos t_{0} )\sum _{j=1}^{n-1}a_{j} '\sin jt .\] 

Consider the set $T_{0} =\left\{\, \theta ,t_{1} ,...,t_{q} \, \right\}$ and corresponding pair of trigonometric polynomials $(C_{\theta } (t),S_{\theta } (t))$, where
\[S_{\theta } (t)=\frac{\cos t-\cos \theta }{(1-\cos \theta ) \sigma -\frac{a_{1} '}{2} } \sum _{j=1}^{n-1}a_{j} '\sin jt \]
with $\sigma =\sum _{j=1}^{n-1}a_{j} ' $. The normalizing factor in the polynomial $S_{\theta } (t)$ is selected from lemma~\ref{lem2}. It is clear that $S_{t_{0} } (t)\equiv S^{0} (t)$.

From Lemma \ref{lem3} it follows that
\[C_{\theta } (\theta )=-\frac{\frac{a_{1} '}{2} }{(1-\cos \theta )\cdot \sigma -\frac{a_{1} '}{2} } ,\; \theta \in \left[\, t_{0} ,\, \pi \, \right).\] 

According to Lemma \ref{lem1}, the value $C^{0} (t_{0} )=-\frac{a_{1} '}{2} $ is negative. Hence $a_{1} '>0$ and $\sigma$ =$\frac{1+\frac{a_{1} '}{2} }{1-\cos t_{0} } >0$. Under these conditions the value of $C_{\theta } (\theta )$ increases as $\theta $ increases and doesn't exceed $-\frac{a_{1} '}{4\sigma -a_{1} '} $. From the continuity of the functions $C(t), S(t)$  it follows that if $0<\theta -t_{0} <\varepsilon $, then $C_{\theta } (\theta )>C^{0} (t_{0} ),\, \, \left|C_{\theta } (t_{j} )-C^{0} (t_{j} )\right|<\delta ,\, j=1,\, \ldots \, ,\, q,\, \, $ $ and\left|C_{\theta } (\pi ),\ldots,C^{0} (\pi )\right|<\delta $ for any small $\varepsilon $. Then 
$$\min \left\{\, \, C_{\theta } (\theta ),\ldots,C_{\theta } (t_{q} ),C_{\theta } (\pi )\, \right\}>\min \left\{\, \, C^{0} (t_{0} ),...,C^{0} (t_{q} ), 
C^{0} (\pi )\, \right\}=C^{0} (t_{0} )$$
at least for sufficiently small positive value of the difference $\theta -t_{0} $. This means that the pair of polynomials $C^{0} (t),S^{0} (t)$ cannot be optimal.

Consider another case:
\[C^{0} (\pi )\le C^{0} (t_{0} )<C^{0} (t_{j} ),\; j=1,...,q.\] 

For polynomials $C_{\theta } (t),\, \, S_{\theta } (t)$ the equality below holds by Lemma \ref{lem4}:
$$C_{\theta } (\pi )=-\frac{(1+\cos \theta )\sigma '+\frac{a_{1} '}{2} }{(1-\cos \theta )\cdot \sigma -\frac{a_{1} '}{2} } ,$$ 
where $\sigma '=-a_{1} '+a_{2} '-\ldots$. If $\sigma '\le -\frac{a_{1} "\sigma }{4\sigma -a_{1} '} $, then $C_{\theta } (\pi )\ge -\frac{a_{1} '}{4\sigma -a_{1} '} >C_{\theta } (\theta ),\; \theta \in (\, t_{0} ,\, \, \pi \, )$. Hence $\sigma '>-\frac{a_{1} '\sigma }{4\sigma -a_{1} '} $. In this case the values $C_{\theta } (\pi ),C_{\theta } (\theta )$ increases as $\theta $ increases. So the pair $(C^{0} (t),S^{0} (t))$ cannot be an optimal one.

\bigskip
Let $C^{0} (t_{0} )=C^{0} (t_{1} )<C^{0} (t_{j} ),\; j=2,\ldots,q, t_{0} \ne t_{1} $. Then as a concequence of Lemma \ref{lem6}, $a_{1} ''=0$. Consider pair of polynimials $(C_{\theta _{1} ,\theta _{2} } (t),S_{\theta _{1} ,\theta _{2} } (t))$, where
$$S_{\theta _{1} ,\theta _{2} } (t)=\frac{(\cos t-\cos \theta _{1} )(\cos t-\cos \theta _{2} )}{(1-\cos \theta _{1} )(1-\cos \theta _{2} )\sigma ''-\frac{a_{2} ''}{4} }\sum _{j=2}^{n-2}a_{j} ''\sin jt .$$ 
Here $\sigma ''=\sum _{j=2}^{n-2}a_{j} '' $. It is clear that $S_{t_{0} ,t_{1} } (t)=S^{0} (t)$.

Then by lemma 5
$$C_{\theta _{1} ,\theta _{2} } (\theta _{j} )=-\frac{\frac{a_{2} ''}{4} }{(1-\cos \theta _{1} )(1-\cos \theta _{2} )\sigma ''-\frac{a_{2} ''}{4} },\qquad j=1,\, 2.$$

Because $C^{0} (t_{0} )<0$, so $a_{2} ''>0,\; \sigma ''>0$. In this case the functions $C_{\theta _{1} ,\theta _{2} } (\theta _{j} )$ increases with respect to $\theta _{1} $ and $\theta _{2} $. Again, the pair $(C^{0} (t),S^{0} (t))$ cannot be the optimal one.

\bigskip

We can make an analogical conclusion in the case when there are more than 2 minimal elements in the set $\left\{\, \, C^{0} (t_{0} ),\ldots,C^{0} (t_{q} ),C^{0} (\pi )\, \right\}$. So it is shown, that the set \textit{T} is an empty set, i.e. for the optimal pair it is necessarily satisfied $S^{0} (t)\ge 0,\; t\in \left[\, 0,\pi \, \right]$. 

\bigskip
The trigonometric polynomial $S^{0} (t)$ can be performed in the form
$$S^{0} (t)=\sin t (\gamma _{1} +2\gamma _{2} \cos t+\ldots+2\gamma _{n} \cos (n-1)t),$$
where $\gamma _{1} =a_{1} +a_{3} +\ldots,$ $\gamma _{2} =a_{2} +a_{4} +\ldots,$ $\gamma _{3} =a_{3} +a_{5} +\ldots,$ $\gamma _{4} =a_{4} +\ldots $.

\noindent There is a bijection between $a_{1} ,\ldots, a_{n} $ and $\gamma _{1} ,\, ...\, ,\, \, \gamma _{n} $.
\noindent  The normalization condition  $\sum _{j=1}^{n}a_{j} =1 $ is equivalent to the equality $\gamma _{1} +\gamma _{2} =1$.

Because $T=\emptyset$, then 
$$\overline{\rho }_{1} =\mathop{\max }\limits_{(a_{1} ,\ldots,a_{n} )\in A_{R} } \left\{\, C(\pi )\, \right\}\, \, =\mathop{\max }\limits_{(a_{1} ,\ldots,a_{n} )\in A_{R} } \left\{-a_{1} +a_{2} -a_{3} +\ldots\right\}.$$
Note that$-a_{1} +a_{2} -a_{3} +\ldots=-\gamma _{1} +\gamma _{2} $.

\noindent It follows from [2] that the polynomial $S^{0} (t)$ is nonnegative implies that $\left|\gamma _{2} \right|\le \cos \frac{\pi }{n+1} \cdot \left|\gamma _{1} \right|$. Then
$$\overline{\rho }_{1} =\mathop{\max }\limits_{\gamma _{1} ,\gamma _{2} } \left\{-\gamma +\gamma _{2} :\gamma _{1} +\gamma _{2} =1,\; \left|\gamma _{2} \right|\le \cos \frac{\pi }{n+1}\left|\gamma _{1} \right|\right\}.$$
The constrained maximum is reached when
$$\gamma _{1}^{0} =\frac{1}{1+\cos \frac{\pi }{n+1} } ,\; \, \, \gamma _{2}^{0} =\frac{\cos \frac{\pi }{n+1} }{1+\cos \frac{\pi }{n+1} } ,$$
And equals to
$$\overline{\rho }_{1} =-\frac{1-\cos \frac{\pi }{n+1} }{1+\cos \frac{\pi }{n+1} } =-tg^{2} \frac{\pi }{2(n+1)} .$$

The polynomial $\frac{S^{0} (t)}{\sin t} =\gamma _{1}^{0} +2\gamma _{2}^{0} \cos t+\ldots+2\gamma _{n}^{0} \cos (n-1)t$ is called nonnegative Feier polynomial, and its coefficients are determined uniquely. Hence coefficients $a_{1}^{0} ,\ldots,a_{n}^{0} $ are defined uniquely: $a_{1}^{0} =\gamma _{1}^{0} -\gamma _{3}^{0},$ $a_{2}^{0} =\gamma _{2}^{0} -\gamma _{4}^{0} ,$ $a_{3}^{0} =\gamma _{3}^{0} -\gamma _{5}^{0} \ldots.$ And they are positive and doesn't depend on $R$ because $\sum _{j=1}^{n}\left|a_{j} \right| =1$.

It means that for all $a_{1},\ldots,a_{n} $ with $\sum _{j=1}^{n}\left|a_{j} \right| =1$, we have the followig inequalities
$$\rho _{1} (a_{1},\ldots,a_{n} )\le \overline{\rho _{1} },\; \rho (a_{1},\ldots,a_{n} )\le \overline{\rho }\le \overline{\rho _{1} }.$$

We also need to show that for function $\rho (a_{1},\ldots,a_{n} )$ the $\sup$ is achieved and equals to $\overline{\rho _{1} }$. Let us consider one-parameter family of trigonometric polynomials
$$S^{\varepsilon } (t)=\frac{a_{1}^{0} +\varepsilon }{1+\varepsilon } \sin t+\frac{a_{2}^{0} }{1+\varepsilon } \sin 2t+\ldots+\frac{a_{n}^{0} }{1+\varepsilon } \sin nt.$$ 
It is clear that $\frac{a_{1}^{0} +\varepsilon }{1+\varepsilon } +\frac{a_{2}^{0} }{1+\varepsilon } +\ldots+\frac{a_{n}^{0} }{1+\varepsilon } =1$ and $S^{\varepsilon } (t)=\frac{1}{1+\varepsilon } S^{0} (t)+\frac{\varepsilon}{1+\varepsilon } \sin t$. For all $t\in (0,\pi )$and $\varepsilon >0$ holds inequality $S^{\varepsilon } (t)>0$. Because $C^{\varepsilon } (\pi )=\frac{1}{1+\varepsilon } \overline{\rho }_{1} -\frac{\varepsilon }{1+\varepsilon } $,  $C^{\varepsilon } (\pi )<\overline{\rho }_{1} $ and $C^{\varepsilon } (\pi )\to \overline{\rho }_{1} $ at $\varepsilon \to 0$. These conditions and the independence of the coefficients of $R$ means that
$$I=\bar{\rho }=\bar{\rho }_{1} =\mathop{\sup }\limits_{\begin{array}{c} {a_{1} ,\ldots,a_{n} } \\ {\sum a_{j} =1 } \end{array}} \left\{\, \rho (a_{1} ,\ldots,a_{n} )\, \right\}=-tg^{2} \frac{\pi }{2(n+1)} .$$ 

This completes the proof of the theorem.\end{proof}

\bigskip

It is possible to find the optimal coefficients $a_{1}^{0} ,\ldots,a_{n}^{0} $ in an explicit way. Indeed, polynomial $\frac{S^{0} (t)}{\sin t} $ is proportional to the Feier polynomial:
$$\frac{S^{0} (t)}{\sin t} =\frac{1}{1+\cos \frac{\pi }{n+1} } +\frac{2\cos \frac{\pi }{n+1} }{1+\cos \frac{\pi }{n+1} } \cos t+\ldots=$$
$$\frac{1-\cos \frac{\pi }{n+1} }{n+1} \cdot \frac{2\cos ^{2} \frac{n+1}2 t}{(\cos t-\cos \frac{\pi }{n+1} )^{2} } =\gamma _{1}^{0} +2\gamma _{2}^{0} \cos t+\ldots\, \, .$$
Hence the coefficients $\gamma _{1}^{0} ,\ldots,\gamma _{n}^{0} ,\; a_{1}^{0} ,\ldots,a_{n}^{0} $ are defined by formulas
$$\gamma _{j}^{0} =\frac{1}{2(n+1)\sin \frac{\pi }{n+1}(1+\cos \frac{\pi }{n+1} )}  \left((n-j+3)\sin \frac{\pi j}{n+1} -(n-j+1)\sin \frac{\pi (j-2)}{n+1} \right),$$
$$a_{j}^{0} =2\cdot tg\frac{\pi }{2(n+1)} \cdot (1-\frac{j}{n+1} )\cdot \sin \frac{\pi j}{n+1} ,\; j=1,...,n.$$
Returning to the problem of maximizing the function \eqref{GrindEQ__2_}, we get
$$\mathop{\max }\limits_{a_{1} ,\ldots,a_{n} } \left\{\, k_{2} (a_{1} ,\ldots,a_{n} )\, \right\}=k_{2} (a_{1}^{0} ,\ldots,a_{n}^{0} )=ctg^{2} \frac{\pi }{2(n+1)} .$$
Note that coefficients $a_{1}^{0} ,\ldots,a_{n}^{0} $ are positive, so
$$\mathop{\max }\limits_{a_{1} ,\ldots,a_{n} } \left\{\, \, k_{1} (a_{1} ,\ldots,a_{n} )\, \right\}=k_{1} (a_{1}^{0} ,\ldots,a_{n}^{0} )=1.$$ 
Finally, $$\mathop{\max }\limits_{\begin{array}{c} {(a_{1} ,\ldots,a_{n} )} \\ {\sum a_{j} =1 } \end{array}} \Phi (a_{1} ,\ldots,a_{n} )=1+ctg^{2} \frac{\pi }{2(n+1)} =\frac{1}{\sin ^{2} \frac{\pi }{2(n+1)} } .$$


\bigskip
\begin{thebibliography}{10}

\bibitem{CY} Chen G., Yu X. \textit{ Chaos control},  Lect. Notes Contr. Inf. Sci. 2003. No. 292.
\bibitem{Fe}  Fejer L. \textit{Ueber trigonometrische polynome}, J. fuer die reine und angew. Math. 1915. Bd. 146. S. 53-82. 

\end{thebibliography}
\end{document}